\documentclass[11pt]{article}  
\usepackage{amsmath}
\usepackage{amssymb}
\usepackage{theorem}
\usepackage{euscript}
\usepackage{pstricks}
\newtheorem{theorem}{Theorem}[section]
\newtheorem{lemma}{Lemma}[section]

\numberwithin{equation}{section}
\title{\sffamily  Adaptive algorithms in sampling recovery 
\author{ 
Dinh D\~ung \\[5mm]
Information Technology Institute,
Vietnam National University, Hanoi\\
144 Xuan Thuy, Cau Giay, Hanoi, Vietnam\\
{\ttfamily dinhdung@vnu.edu.vn}\\[4mm]}}
 \date{\ttfamily February 26, 2011 - Version 0.2}
\def\II{{\mathbb I}}

\def\ZZ{{\mathbb Z}}
\def\RR{{\mathbb R}}
\def\NN{{\mathbb N}}

\def\Pp{{\mathcal P}}

\def\Cc{{\mathcal C}}
\def\sgn{\operatorname{sgn}}
\def\dimp{\dim_{\operatorname{p}}}
\begin{document}
\maketitle

\begin{abstract}
We study optimal algorithms in
 adaptive  sampling recovery of smooth functions defined on the unit $d$-cube ${\II}^d:= [0,1]^d$. 
The recovery error is measured in the quasi-norm $\|\cdot\|_q$ of $L_q := L_q(\II^d)$.
For $B$  a subset in $L_q,$ we  define a sampling recovery algorithm 
with the free choice of sample points and recovering functions from $B$ as follows. For each 
$f$ from the quasi-normed Besov space $B^\alpha_{p,\theta}$, 
 we choose $n$ sample points. This choice defines $n$ sampled 
values. Based on these sample points and sampled values, we choose a function from 
$B$ for recovering $f$.  The choice of $n$ sample points and a recovering function from $B$ for 
each $f \in B^\alpha_{p,\theta}$ defines a $n$-sampling algorithm $S_n^B$ by functions in $B$.
We suggest a new approach to investigate the optimal adaptive sampling recovery by 
$S_n^B$ in the sense of continuous non-linear $n$-widths which is
related to $n$-term approximation.
If $\Phi = \{\varphi_k\}_{k \in K}$ 
is a family of elements in $L_q$, let $\Sigma_n(\Phi)$ be the non-linear set
of linear combinations of $n$ free terms from $\Phi,$ that is
$
\Sigma_n(\Phi):= 
\{\, \varphi = \sum_{j=1}^n a_j \varphi_{k_j}: \ k_j \in K \, \}.
$
Denote by ${\mathcal G}$ the set of all families $\Phi$ in $L_q$ such that 
the intersection of $\Phi$ with any finite dimensional subspace in $L_q$ is 
a finite set, and by $\Cc(B^\alpha_{p,\theta}, L_q)$ the set of all continuous
mappings from $B^\alpha_{p,\theta}$ into $L_q$.  
We define the quantity
\begin{equation*}
\nu_n(B^\alpha_{p,\theta},L_q)
\ := \ \inf_{\Phi \in {\mathcal G}} \ \inf_{S_n^B \in \Cc(X, L_q): \ B= \Sigma_n(\Phi)} \ 
\sup_{\|f\|_{B^\alpha_{p,\theta}} \le 1} \ \|f -  S_n^B(f)\|_q.
\end{equation*}
Let $0 < p,q , \theta \le \infty $ and $\alpha > d/p$. 
Then we prove the asymptotic order
\begin{equation*} 
\nu_n(B^\alpha_{p,\theta},L_q) 
\ \asymp \ 
n^{- \alpha / d}. 
\end{equation*}
We also obtained the asymptotic order of quantities 
of optimal recovery by $S_n^B$ in terms of best 
$n$-term approximation as well of other non-linear $n$-widths.


  
\end{abstract}
%%%%%%%%%%%%%%%%%%%%%%%%%%%%%%%%%%%%%%%%%%%%%%%%%%%%%%%%%%%%%%%%%%%%%%%%%%%5

\section{Introduction} 

The purpose of the present paper is to investigate optimal algorithms in 
adaptive sampling recovery of functions defined on the unit $d$-cube 
${\II}^d:= [0,1]^d.$ Let $L_q :=L_q({\II}^d), \ 0 < q \le \infty,$
denote the quasi-normed space 
of functions on ${\II}^d$ with the usual $q$th integral quasi-norm 
$\|\cdot\|_q$ for $0 < q < \infty,$ and 
the normed space $C({\II}^d)$ of continuous functions on ${\II}^d$ with 
the max-norm $\|\cdot\|_{\infty}$ for $q = \infty$. For $0 < p,\theta, q \le \infty$ and $\alpha > 0$,
let $B^\alpha_{p,\theta}$ be the quasi-normed Besov space with smoothness $\alpha$, equipped 
with the quasi-norm $\|\cdot\|_{B^\alpha_{p,\theta}}$ (see Section \ref{Preliminary} for the definition).
We consider problems of adaptive sampling recovery of functions from 
$B^\alpha_{p,\theta}$. The recovery error will be measured in the quasi-norm $\|\cdot\|_q.$ 

We first recall some well-known non-adaptive sampling recovery algorithms. 
Let $X$ be a quasi-normed space of functions defined on $\II^d$, such that the linear functionals
$f \mapsto f(x)$ are continuous for any $x \in \II^d$. We assume that $X \subset L_q$ and the 
embedding $\operatorname{Id}: X \to L_q$ is continuous, where $\operatorname{Id}(f):=f$. 
Suppose that $f$ is a function in $X$ and $\xi_n = \{x^k\}_{k =1}^n$ are  $n$  points 
in ${\II}^d.$ We want to approximately recover $f$
 from the sampled values 
$f(x^1), f(x^2), ..., f(x^n)$.
A classical linear 
sampling algorithm of recovery is
\begin{equation} \label{L(f)}
L_n(f) \ = \ L_n(\Phi_n,\xi_n,f) 
:= \ \sum_{k =1}^n f(x^k)\varphi_k,
\end{equation}
where 
$\Phi_n = \{\varphi_k\}_{k =1}^n$ are given  
$n$ functions in $L_q$. 
 A more general sampling algorithm of recovery can be defined as 
\begin{equation} \label{R(f)}
R_n(f) = R_n(H_n,\xi_n, f):= H_n(f(x^1),...,f(x^n)), 
\end{equation}
where $H_n$ is a given mapping from 
${\mathbb R}^n$ to $L_q$. Such a sampling algorithm is, in general, non-linear.
Denote by $SX$ the unit ball in the quasi-normed space $X$. 
To study  optimal sampling algorithms of recovery for $f \in X$ from 
$n$ their values by algorithms of the form \eqref{R(f)},  one can use the quantity
\begin{equation*} 
g_n(X,L_q) 
\ := \ \inf_{H_n,\xi_n} \  \sup_{f \in SX} \, \|f - R_n(H_n,\xi_n,f)\|_q, 
\end{equation*}
where the infimum is taken over all sequences $\xi_n = \{x^k\}_{k =1}^n$ 
and all mappings $H_n$ from $\RR^n$ into $L_q$. 

We use the notations: $x_+ := \max(0,x)$ for $x \in \RR;$  
$A_n(f) \ll B_n(f)$ if $A_n(f) \le CB_n(f)$ with 
$C$ an absolute constant not depending on $n$ and/or $f \in W,$ and 
$A_n(f) \asymp B_n(f)$ if $A_n(f) \ll B_n(f)$ and $A_n(f) \ll A_n(f).$ 
It is known the following result 
(see \cite{Di1}, \cite{K}, \cite{No}, \cite{NoT}, \cite{Te} and references there).
If $0 < p,\theta, q \le \infty$ and $\alpha > d/p$, 
then there is a linear sampling recovery method $L_n^*$
of the form \eqref{L(f)} such that  
\begin{equation} \label{asympG}
g_n(B^\alpha_{p,\theta}, L_q) 
\ \asymp \ 
\sup_{f \in SB^\alpha_{p,\theta}} \ \|f - L_n^*(f)\|_q 
\ \asymp \ 
n^{- \alpha /d + (1/p - 1/q)_+}. 
\end{equation}
This result says that the linear sampling algorithm
$L_n^*$ is asymptotically optimal in the sense that
any sampling algorithm $R_n$
of the form \eqref{R(f)} does not give the rate of convergence better than 
$L_n^*.$

Sampling algorithms of recovery of the form \eqref{R(f)} 
which may be linear or non-linear
 are non-adaptive, i.e., 
 the points $\xi_n = \{x^k\}_{k =1}^n$ 
at which the values $f(x^1),...,f(x^n)$ are sampled,
and the method for construction of recovering functions
are the same for all functions $f \in X.$  Let us introduce a setting 
of adaptive sampling recovery.  

If $B$ is a subset in $L_q$, we define a sampling algorithms of recovery 
with the free choice of sample points and recovering functions from $B$ as follows. For each 
$f \in X$ we choose a set of $n$ sample points. This choice defines a collection of $n$ sampled 
values. Based on the information of these sampled values, we choose a function from 
$B$ for recovering $f$. 
The choice of $n$ sample points and a recovering function from $B$ for 
each $f \in X$ defines a sampling algorithms of recovery $S_n^B$ by functions in $B$. 
More precisely, a formal definition of $S_n^B$ is given as follows. Denote by 
${\mathcal I}^n$ the set of subsets $\xi$ in ${\II}^d$ of cardinality at most $n$,
${\mathcal V}^n$ the set of subsets $\eta$ in $\RR \times {\II}^d$ of cardinality at most $n$. 
Let $T_n$ be a mapping from $X$ into ${\mathcal I}^n$. 
Then $T_n$ generates an
{\em $n$-sampling operator} $I_n$ from $X$ into ${\mathcal V}^n$ which is defined as follows. If 
$T_n(f)= \{x^1,...,x^n\}$ then $I_n(f)= \{(f(x^1),x^1),...,(f(x^n),x^n)\}$. 
Let $P_n^B$ a mapping from ${\mathcal V}^n$ into $B$.
Then the pair $(I_n, P_n^B)$ generates the mapping $S_n^B$ from 
$X$ into $B$, by the formula
\begin{equation} \label{S_n^B(f)}
S_n^B(f):=  P_n^B(I_n(f)), 
\end{equation} 
which defines a {\em $n$-sampling algorithm} with the
free choice of $n$ sample points and approximant from $B$. 
We call the mapping $P_n^B$ a {\em recovering operator}.

Clearly, a linear sampling algorithm $L_n(\Phi_n,\xi_n,\cdot)$ defined in \eqref{L(f)} 
is a particular case of $S_n^B$. We are interested in adaptive $n$-sampling algorithms $S_n^B$ of special form
which are an extension of $L_n(\Phi_n,\xi_n,\cdot)$
to an $n$-sampling algorithm with the
free choice of $n$ sample points and $n$ functions $\Phi_n = \{\varphi_k\}_{k =1}^n$ for each $f \in X$. 
To this end we let $\Phi = \{\varphi_k\}_{k \in K}$ 
be a family of elements in $L_q$, and consider the non-linear set $\Sigma_n(\Phi)$
of linear combinations of $n$ free terms from $\Phi,$ that is
\begin{equation*}
\Sigma_n(\Phi):= 
\{\, \varphi = \sum_{j=1}^n a_j \varphi_{k_j}: \ k_j \in K \, \}.
\end{equation*}
Then for $B=\Sigma_n(\Phi)$,  an $n$-sampling algorithm $S_n^B$
is of the following form
\begin{equation} \label{eq:[S_n^Sigma]}
S_n^B(f) 
 \ = \  
\sum_{k \in Q(\eta)} a_k(\eta) \varphi_k,
\end{equation}
where $\varphi_k \in \Phi$, $\eta = I_n(f)$,  $Q(\eta) \subset K$ with
 $|Q(\eta)| \le n$ and $a_k$ are functions on ${\mathcal V}^n$.

We want to choose an 
$n$-sampling algorithm $S_n^B$ so that the error of this recovery 
$\|f -  S_n^B(f)\|_q$ is as smaller as possible.
Clearly, such an efficient choice  
 should be adaptive to $f$. 
To investigate the optimality of (non-continuous) adaptive recovery of functions $f$ from 
the quasi-normed space $X$ by $n$-sampling algorithms $S_n^B$ of the form \eqref{eq:[S_n^Sigma]}, 
we introduce the quantity $s_n(X,\Phi,L_q)$  as follows:
\begin{equation} \label{def:[s_n]}
s_n(X,\Phi,L_q)
\ := \  \ \inf_{S_n^B: \ B= \Sigma_n(\Phi)} \ \sup_{f \in SX} \ \|f -  S_n^B(f)\|_q.
\end{equation}

The definition \eqref{def:[s_n]} corrects a definition of $s_n(X,\Phi,L_q)$ 
which has been introduced and denoted by 
$\nu_n(SX,\Phi)_q$ and $s_n(SX,\Phi)_q$ in \cite{Di6} and \cite{Di7}, respectively.
 The quantity $s_n(X,\Phi,L_q)$ is directly related to non-linear $n$-term approximation.
We refer the reader to \cite{D}, \cite{Te1} for surveys on various aspects in the last direction.

The quantity $s_n(X,\Phi,L_q)$ depends on 
the family $\Phi$ and therefore, is not absolute in the sense
of $n$-widths or optimal algorithms. 
We suggest an approach to investigate the optimal adaptive sampling recovery by 
$S_n^B$ in the sense of continuous non-linear $n$-widths which is
related to $n$-term approximation too.
Namely, we consider the optimality in the restriction with only $n$-sampling algorithms of recovery 
$S_n^B$ of the form \eqref{eq:[S_n^Sigma]} and with a continuity assumption on them. 
Continuity assumptions on  approximation and recovery algorithms have their origin 
in the very old Alexandroff $n$-width which characterizes best continuous approximation algorithm
by $n$-dimensional topological complexes \cite{A} (see also \cite{Ti} for details). 
Later on, (continuous) manifold $n$-width was introduced by 
in \cite{DHM}, \cite{M}, and investigated in \cite{DY}, \cite{DKLT}, \cite{DiV}, \cite{Di2}, \cite{Di3},\cite{Di4}. 
Several continuous $n$-widths based on continuous algorithms of $n$-term    
approximation, were introduced and studied in \cite{Di2}, \cite{Di3},\cite{Di4}. 
The continuity assumption is quite natural: the closer objects
are the closer their reconstructions should be. A first look seems that a continuity restriction may
decrease the choice of approximants. However, in most cases it does not weaken the rate of the corresponding
approximation. Continuous and non-continuous algorithms of nonlinear approximation give the same asymptotic order. 
This motivate us to consider continuous $n$-sampling algorithms of recovery $S_n^B$. Since we assume that 
functions to be recovered are living in the quasi-normed space $X$ 
and the recovery error is measured in the
quasi-normed space $L_q$, the requirement that $S_n^B \in \Cc(X, L_q)$ is quite reasonable. 
(Here and in what follows, $\Cc(X,Y)$ denotes the set of all continuous
mappings from $X$ into $Y$ for the quasi-metric spaces $X,Y$). This leads to the following definition. 

Denote by ${\mathcal G}$ the set of all families $\Phi$ in $L_q$ such that 
the intersection of $\Phi$ with any finite dimensional subspace in $L_q$ is 
a finite set. We define the quantity
\begin{equation*}
\nu_n(X,L_q)
\ := \ \inf_{\Phi \in {\mathcal G}} \ \inf_{S_n^B \in \Cc(X, L_q): \ B= \Sigma_n(\Phi)} \ 
\sup_{f \in SX} \ \|f -  S_n^B(f)\|_q.
\end{equation*}
 The restriction $\Phi \in {\mathcal G}$ 
in the definition of $\nu_n(X,L_q)$ is minimal and natural for all well-known approximation systems.

 Another way to study optimal adaptive (non-continuous) $n$-sampling algorithms of recovery $S_n^B$
in the sense of nonlinear $n$-widths has been proposed in \cite{Di6}, \cite{Di7}. 
In this approach, $B$ is required to have a finite capacity 
which is measured by their cardinality or pseudo-dimension. 
Given a family ${\mathcal B}$ of subsets in $L_q$, we consider optimal 
sampling recoveries by $B$ from ${\mathcal B}$ in terms of the quantity
\begin{equation} \label{R_n}
R_n(W, {\mathcal B})_q := \  \inf_{B \in {\mathcal B}} \ \inf_{S_n^B} \ \sup_{f \in W} \ \|f -  S_n^B(f)\|_q.
\end{equation}
We assume a restriction on the sets $B \in {\mathcal B}$, requiring that they 
should have, in some sense, a finite capacity. In the present paper, the capacity of $B$  
is measured by its cardinality or pseudo-dimension. This reasonable restriction 
would provide nontrivial lower bounds of asymptotic order of 
$R_n(W, {\mathcal B})_q $ for well known function classes $W$.
Denote $R_n(W, {\mathcal B})_q$ by $e_n(W)_q$ 
 if ${\mathcal B}$ in \eqref{R_n} is the family of all subsets $B$
in $L_q$ such that $|B| \le 2^n$, where $|B|$ denotes the cardinality
of $B,$ and by $r_n(W)_q$ 
 if ${\mathcal B}$ in \eqref{R_n} is the family of all subsets $B$ in 
$L_q$ of pseudo-dimension at most $n$. 
The definition \eqref{R_n} corrects definitions of $e_n(W)_q$ and $r_n(W)_q$ 
introduced in \cite{Di7}.

 The quantity $e_n(W)_q$ is related to the entropy $n$-width (entropy number)
$\varepsilon_n(W)_q$  which is the functional inverse of the
classical $\varepsilon$-entropy introduced by Kolmogorov and Tikhomirov \cite{KT}. 
The quantity $r_n(W)_q$ is related to the non-linear $n$-width 
$\rho_n(W)_q$ introduced recently by Ratsaby and Maiorov \cite{RM1}. (See the definition 
of $\varepsilon_n(W)_q$ and $\rho_n(W)_q$ in Section \ref{NC_recovery}).

The pseudo-dimension of a  set $B$ of real-valued functions on 
a set $\Omega$, is defined as follows. 
For a real number~$t$, let $\sgn(t)$  be $1$ 
for $t>0$ and $-1$ otherwise. For $x \in {\RR}^n$, let  
$\sgn(x) = (\sgn(x_1),\sgn(x_2),...,\sgn(x_n))$.
The pseudo-dimension of $B$ is defined as the largest integer $n$ such that 
there exist points $a^1, a^2, \dots, a^n$ in $\Omega$ and 
$b \in {\RR}^n$ such that the cardinality of the set 
\begin{equation*}
 \{\,\sgn(y): \ y = (f(a^1) + b_1, f( a^2) + b_2, \dots, f( a^n) + b_n), 
\ f \in B \}
\end{equation*}
is $2^n$. If $n$ is arbitrarily large, then the the pseudo-dimension 
of $B$ is infinite.  Denote the pseudo-dimension of~$B$ by
$\dimp(B)$. 
The notion of pseudo-dimension was
introduced by Pollard \cite{P} and later Haussler \cite{H1} as an extention of
the  VC-dimension \cite{VC}, suggested by Vapnik-Chervonekis for sets of indicator functions.
The pseudo-dimension and VC-dimension measure the
capacity of a set of functions and are related to its $\varepsilon$-entropy  
(see also \cite{RM1}, \cite{RM2}). 
If $B$ is a $n$-dimensional linear manifold of real-valued functions on $\Omega$, 
then  $\dimp(B) = n$ (see \cite{H1}). 

We say that $ p,q , \theta, \alpha$ satisfy Condition \eqref{Condition} if
\begin{equation} \label{Condition}
\begin{aligned}
 0 <  p,q , & \theta \le \infty , \ 0 < \alpha < \infty, 
  \ {\rm and \ there \ holds \  
 one \ of \ the \ following \ restrictions:} \\
  {\rm (i)} & \  \alpha > d/p; \\
 {\rm (ii)} & \ \alpha = d/p, \ \theta \le \min(1,p), \ p,q < \infty.
 \end{aligned}
\end{equation}
Let ${\bf M}$ be the set of B-splines which are the tensor product of integer translated dilations 
of the centered cardinal spline of order $2r$ (see the definition in Section \ref{Preliminary}). 

The main results of the present paper are read as follows.

 \begin{theorem} \label{ThNU} 
Let  $ p,q , \theta, \alpha$ satisfy the Condition \eqref{Condition} and $\alpha < 2r$.
 Then for the $d$-variable Besov space 
$B^\alpha_{p,\theta}$,  there is 
the following asymptotic order
\begin{equation} \label{Asymp:[nu]}
s_n(B^\alpha_{p,\theta},{\bf M},L_q) 
\ \asymp \ 
\nu_n(B^\alpha_{p,\theta},L_q) 
\ \asymp \ 
r_n(SB^\alpha_{p,\theta})_q 
\ \asymp \ 
e_n(SB^\alpha_{p,\theta})_q 
\ \asymp \ 
n^{- \alpha / d} . 
\end{equation}
\end{theorem}

Comparing this asymptotic order with \eqref{asympG}, we can see that 
for $0 <  p < q \le \infty,$ the asymptotic order of optimal adaptive
sampling recovery in terms of the quantities $s_n$, $\nu_n$, $e_n$ and $r_n$,
is better than the asymptotic order of any non-adaptive $n$-sampling algorithm of 
 recovery of the form \eqref{R(f)}. 

To prove the upper bound for \eqref{Asymp:[nu]}, 
we use a B-spline quasi-interpolant representation of functions in 
the Besov space $B_{p, \theta}^\alpha$ 
associated with some equivalent discrete quasi-norm \cite{Di6}, \cite{Di7}. 
On the basis of this representation we construct corresponding asymptotically optimal 
$n$-sampling algorithms of recovery  which give the upper bound for \eqref{Asymp:[nu]}. 
The lower bound of \eqref{Asymp:[nu]} is established by the lower estimating of the smaller 
related $n$-widths and the quantity of $n$-term approximation. 

The paper is organized as follows.
 
In Section \ref{Preliminary}, we give a definition of quasi-interpolant 
for functions on ${\II}^d$, describe a B-spine quasi-interpolant representation 
for Besov spaces $B^\alpha_{p,\theta}$  with 
a discrete quasi-norm in terms of the coefficient functionals. 
The proof of the asymptotic order of $\nu_n(B^\alpha_{p,\theta},L_q)$ 
in Theorem \ref{ThNU} is given in Sections \ref{C_recovery} and \ref{Lowerbounds}.
More precisely, in Section \ref{C_recovery}, we construct asymptotically optimal 
adaptive $n$-sampling algorithms of recovery  
which give the upper bound for $\nu_n(B^\alpha_{p,\theta},L_q)$ (Theorem \ref{ThUpperBound}). 
In Section \ref{Lowerbounds} we prove the lower bound for $\nu_n(B^\alpha_{p,\theta},L_q)$
(Theorem \ref{[Thnu>>]}). In Section \ref{NC_recovery}, we prove 
the asymptotic order of  $s_n(B^\alpha_{p,\theta},{\bf M},L_q)$, 
$r_n(SB^\alpha_{p,\theta})_q$ and $e_n(SB^\alpha_{p,\theta})_q$ 
in Theorem \ref{ThNU}.

%%%%%%%%%%%%%%%%%%%%%%%%%%%%%%%%%%%%%%%%%%%%%%%%%%%%%%%%%%%%%%%%%%%%%%%%%%%%%%%%%%%%%%

\section{Preliminary background} 
\label{Preliminary}

For a given natural number $r,$ let  $M$ be the   
centered B-spline of even order $2r$ with support $[-r,r]$ and 
knots at the integer points $-r,...,0,...,r$ and define the B-spline 
\begin{equation*}
M_{k,s}(x):= \ M(2^k x - s), 
\end{equation*}
for a non-negative integer $k$ and $s \in \ZZ.$
To get the $d$-variable B-spline $M_{k,s}$ for a non-negative integer $k$ and $s \in {\ZZ}^d$, we let
\begin{equation*}
M(x):= \ \prod_{i=1}^d M(x_i), \ x = (x_1,x_2,...,x_d), 
\end{equation*}
and 
\begin{equation*}
M_{k,s}(x):= \ M(2^k x - s). 
\end{equation*}
 Denote by ${\bf M}$ the set of 
all $M_{k,s}$
which do not vanish identically on  ${\II}^d.$ 

Let $\Lambda = \{\lambda(j)\}_{j \in P^d(\mu)}$ be a finite even sequence, i.e., 
$\lambda(-j) = \lambda(j),$ where $P^d(\mu):= \{j \in  {\ZZ}^d: \ |j_i| \le \mu, \ i=1,2,...,d \}.$ 
We define the linear operator $Q$ for functions $f$ on ${\RR}^d$ by  
\begin{equation} \label{Q}
Q(f,x):= \ \sum_{s \in {\ZZ}^d} \Lambda (f,s)M(x-s), 
\end{equation} 
where
\begin{equation} \label{Lambda}
\Lambda (f,s):= \ \sum_{j \in P^d(\mu)} \lambda (j) f(s-j).
\end{equation}
The operator $Q$ is bounded in ${C({\RR}^d)}$ and 
\begin{equation*}
\|Q(f)\|_{C({\RR}^d)} \le \|\Lambda \|\|f\|_{C({\RR}^d)}  
\end{equation*}
for each $f \in C({\RR}^d),$ where
\begin{equation*}
\|\Lambda \|= \ \sum_{j \in P^d(\mu)} | \lambda (j) |. 
\end{equation*}
Moreover, $Q$ is local in the following sense. There is 
a positive number $\delta > 0$ such that for  any  
$f \in C({\RR}^d)$ and $x \in {\RR}^d,$ $Q(f,x)$ 
depends only on the value $f(y)$ at a finite number of points 
$y$ with $|y_i - x_i|\le \delta, \ i=1,2,...d.$ 
We will require $Q$ to reproduce  the space 
$\Pp^d_{2r-1}$ of polynomials of order at most $2r - 1$ in each variable $x_i$, that is,  
\begin{equation*} 
Q(p) \ = \ p, \ p \in \Pp^d_{2r-1}. 
\end{equation*}
An operator $Q$ of the form \eqref{Q}--\eqref{Lambda} reproducing $\Pp^d_{2r-1}$, is called 
a {\it quasi-interpolant in} $C({\RR}^d).$ 

There are many ways to construct  quasi-interpolants.
 A method of construction via 
Neumann series was suggested by Chui and Diamond \cite{CD} 
(see also \cite[p. 100--109]{C}). De Bore and Fix \cite{BF} 
introduced another quasi-interpolant 
based on the values of derivatives. 
The reader can see also the books \cite{C}, \cite{BHR} 
for surveys  on quasi-interpolants.  
The most important cases of $d$-variate quasi-interpolants $Q$ are those where
the functional $\Lambda$ is the tensor product of such $d$ univariate functionals. 
Let us give some examples of univariate quasi-interpolants. 
The simplest example is 
a piecewise linear quasi-interpolant is defined for $r=1$ by
\begin{equation*} 
Q(f,x):= \ \sum_{s \in \ZZ} f(s) M(x-s), 
\end{equation*} 
where $M$ is the   
symmetric piecewise linear B-spline with support $[-1,1]$ and 
knots at the integer points $-1, 0, 1$.
This quasi-interpolant is also called {nodal} and directly related 
to the classical Faber-Schauder basis \cite{Di11}. 
Another example is the cubic quasi-interpolant defined for $r=2$ by
\begin{equation*} 
Q(f,x):= \ \sum_{s \in \ZZ} \frac {1}{6} \{- f(s-1) + 8f(s) - f(s+1)\} M(x-s), 
\end{equation*} 
where $M$ is the symmetric cubic B-spline with support $[-2,2]$ and 
knots at the integer points $-2, -1, 0, 1, 2$.

Let $\Omega = [a,b]^d$ be a $d$-cube in ${\RR}^d.$ Denote by  
$L_p(\Omega)$ the quasi-normed space 
of functions on $\Omega$ with the usual $p$th integral quasi-norm 
$\|\cdot\|_{p,\Omega}$ for $0 < p < \infty,$ and 
the normed space $C(\Omega)$ of continuous functions on $\Omega$ with 
the max-norm $\|\cdot\|_{\infty,\Omega}$ for $p = \infty$.
If $\tau$ be a number such that $0 < \tau \le \min (p,1),$ then for any sequence of 
functions $\{ f_k \}$ there is the inequality
\begin{equation} \label{ineq:IneqSumf_k}
\left\| \sum 
  f_k  \right\|_{p,\Omega}^{\tau} 
\  \le \ 
 \sum 
 \| f_k \|_{p,\Omega}^{\tau}.
\end{equation}

We introduce 
Besov spaces of smooth functions 
and give necessary knowledge of them. The reader can read this and more details
about Besov spaces in the books \cite{BIN}, \cite {N}, \cite{DL}.
Let
\begin{equation*}
\omega_l(f,t)_p:= \sup_{|h| < t}\|\Delta_h^l f\|_{p,{\II}^d(lh)}
\end{equation*} 
be the $l$th modulus of smoothness of $f$ where 
${\II}^d(lh):= \{ x: x, x + lh \in {\II}^d \} ,$  
and the $l$th difference $\Delta_h^lf$ is defined by 
\begin{equation*}
\Delta_h^lf(x) := \
\sum_{j =0}^l (-1)^{l - j} \binom{l}{j} f(x + jh).
\end{equation*}
For $0 <  p, \theta \le \infty$ 
and  $0 < \alpha < l,$ the Besov space 
$B_{p, \theta}^\alpha $ is the set of  functions $f \in L_p$ 
for which the Besov quasi-semi-norm 
$|f|_{B_{p, \theta}^\alpha}$ is finite. 
The Besov quasi-semi-norm $|f|_{B_{p, \theta}^\alpha }$ is given by
\begin{equation*} \label{BesovSeminorm}
|f|_{B_{p, \theta}^\alpha }:= 
\begin{cases}
 \ \left(\int_0^\infty \{ t^{-\alpha}
\omega_l(f,t)_p \}^ \theta dt/t \right)^{1/\theta}, 
& \theta < \infty, \\
  \sup_{t > 0} \ t^{-\alpha}\omega_l(f,t)_p,  & \theta = \infty.
\end{cases}
\end{equation*}
The  Besov quasi-norm is defined by
\begin{equation*} \label{BesovQuasinorm}
B(f) \ = \ \|f\|_{B_{p, \theta}^\alpha }
:= \ 
 \|f\|_p + |f|_{B_{p, \theta}^\alpha }. 
\end{equation*}

 If $\{f_k\}_{k=0}^\infty$ is a sequence 
whose component functions $f_k$ are in 
$L_p,$ for $0 < p, \theta \le \infty$ and $\beta \ge 0$ we use the 
$b_\theta^\beta(L_p)$ ``quasi-norms"
\begin{equation*}
\ \|\{ f_k\}\|_{b_\theta^\beta(L_p)}
\ := \ 
 \biggl(\sum_{k=0}^\infty 
 \{2^{\beta k}\| f_k \|_p\}^\theta \biggl)^{1/\theta} 
\end{equation*}
with the usual change to a supremum when $\theta = \infty.$ 
When $\{f_k\}_{k=0}^\infty$ is a positive sequence, we replace 
$\| f_k \|_p$ by $|f_k|$ and denote the corresponding quasi-norm by 
$\|\{ f_k\}\|_{b_\theta^\beta}.$
 
For the Besov space $B^\alpha_{p,\theta},$ there is the following
quasi-norm equivalence
\begin{equation*} 
B(f) 
\ \asymp \ B_1(f)
\ := \
\|\{\omega_l(f,2^{-k})_p\}\|_{b_\theta^\alpha} \ + \ \|f\|_p. 
\end{equation*}

In the present paper, we study the sampling recovery of functions from the Besov space $B^\alpha_{p,\theta}$
with some restriction on the smoothness $\alpha$. Namely, we assume that $\alpha > d/p$. This inequality
 provides the compact embedding of $B^\alpha_{p,\theta}$ into $C({\II}^d)$.
 In addition, we also consider the restriction
$\alpha = d/p$ and $\theta \le \min(1, p)$ which is a sufficient condition for  
the continuous embedding of $B^\alpha_{p,\theta}$ into $C({\II}^d)$.
In both these cases, $B^\alpha_{p,\theta}$ can be considered as a subset in $C(\II^d)$. 

If $Q$ of is a quasi-interpolant of the form 
\eqref{Q}--\eqref{Lambda}, for $h > 0$ and a function $f$ on ${\RR}^d$, 
we define the operator $Q(\cdot;h)$ by
\begin{equation*}
Q(f;h) = \ \sigma_h \circ Q \circ \sigma_{1/h}(f),
\end{equation*}
where $\sigma_h(f,x) = \ f(x/h)$.
By definition it is easy to see that 
\begin{equation*}
Q (f,x;h)= \ 
\sum_{k}\Lambda (f,k;h)M(h^{-1}x-k),
\end{equation*}
where 
\begin{equation*}
\Lambda (f,k;h):= 
 \ \sum_{j \in P^d(\mu)} \lambda (j) f(h(k-j)).
\end{equation*}
The operator $Q(\cdot;h)$ has the same properties as $Q$: it is a local bounded linear 
operator in ${\RR}^d$ and reproduces the polynomials from $\Pp_{2r-1}^d.$ 
Moreover, it gives a good approximation of smooth functions \cite[p. 63--65]{BHR}.
We will also call it a quasi-interpolant for $C({\RR}^d).$

The quasi-interpolant $Q(\cdot;h)$ is not defined for 
a function $f$ on ${\II}^d,$ and therefore, not appropriate for an 
approximate sampling recovery of $f$ from 
its sampled values at points in ${\II}^d.$ 
An approach to 
construct a quasi-interpolant for a function on ${\II}^d$ is to extend it 
by interpolation Lagrange polynomials. This approach has been proposed in 
\cite{Di6} for the univariate case. Let us recall it.

For a non-negative integer $m,$ 
we put $x_j = j2^{-m},  j \in \ZZ.$ If $f$ is a function on $\II,$ let  
$U_m(f)$ and $V_m$ 
be the $(2r - 1)$th Lagrange polynomials interpolating $f$
at the $2r$ left end points $x_0, x_1,..., x_{2r-1},$ and 
$2r$ right end points 
$x_{2^m - 2r + 1}, x_{2^m - 2r + 3},..., x_{2^m},$ 
of the interval $\II,$ respectively. 
The function $f_m$ is defined as an extension of 
$f$ on $\RR$ by the formula
\begin{equation*} 
f_m (x):= \
\begin{cases}
U_m(f,x), \ & x < 0, \\
f(x), \ & 0 \le x \le 1, \\
V_m(f,x), \ & x >1.
\end{cases}
\end{equation*} 
Obviously, if $f$ is continuous on $\II$, then $f_m$ is a continuous function on $\RR.$ 
Let $Q$ be a quasi-interpolant of the form \eqref{Q}-\eqref{Lambda} in $C({\RR}).$ 
We introduce the operator $Q_m$ by putting
\begin{equation*}
Q_m(f,x) := \ Q(f_m,x;2^{-m}), \  x \in \II, 
\end{equation*}
for a function $f$ on $\II$. By definition we have 
\begin{equation*} 
Q_m(f,x)  \ = \ 
\sum_{s \in J(m)} a_{m,s}(f)M_{m,s}(x), \ \forall x \in \II, 
\end{equation*}
where $J(m) := \ \{s \in \ZZ:\ -r < s <   2^m + r \}$ 
is the set of $s$ for which $M_{m,s}$ 
do not vanish identically on  $\II,$ and
\begin{equation*} 
a_{m,s}(f):= \ \Lambda(f_m,s;{2^{-m}}) 
= \   
\sum_{|j| \le \mu} \lambda (j) f_m(2^{-m}(s-j)).
\end{equation*}
The multivariate operator $Q_m$ is defined for functions $f$ on ${\II}^d$ by  
\begin{equation*} %\label{Q_mD}
Q_m(f,x)  \ := \ 
\sum_{s \in J(m)} a_{m,s}(f)M_{m,s}(x), \quad \forall x \in {\II}^d, 
\end{equation*}
where $J^d(m):= \ \{s \in {\ZZ}^d:\ \ - r < s_i < 2^m + r, \ i=0,1,...,d\}$ 
is the set of $s$ for which $M_{m,s}$ 
do not vanish identically on  ${\II}^d,$ and
\begin{equation} \label{a_{m,s}D}
a_{m,s}(f) 
\ = \   
a_{m,s_1}((a_{m,s_2}(...a_{m,s_d}(f))),
\end{equation}
where the univariate functional
$a_{m,s_i}$ is applied to the univariate function $f$ by considering $f$ as a 
function of  variable $x_i$ with the other variables held fixed. 
Moreover, the number of the terms in $Q_m(f)$ is of the size 
$\approx 2^{dm}.$

The operator $Q_m$ is a local bounded linear 
mapping in $C({\II}^d)$  and reproducing $\Pp_{2r-1}^d.$   In particular,
\begin{equation} \label{ineq:Boundedness}
\|Q_m(f)\|_{C({\II}^d)}  \le C \|\Lambda \| \|f\|_{C({\II}^d)} 
\end{equation}
for each $f \in C({\II}^d),$ with a constant $C$ not depending on $m,$ 
and, 
\begin{equation} \label{eq:Reproducing^d}
Q_m(p^*) \ = \ p, \ p \in \Pp_{2r-1}^d, 
\end{equation}
where $p^*$ is the restriction of $p$ on ${\II}^d.$ 
The multivariate operator $Q_m$ is called 
{\it a quasi-interpolant in }$C({\II}^d).$ 

From \eqref{ineq:Boundedness} and \eqref{eq:Reproducing^d} we can see that
\begin{equation} \label{ConvergenceQ_m(f)}
\|f - Q_m(f)\|_{C({\II}^d)}  \to 0 , \ m  \to  \infty.
\end{equation}

Put ${\bf M}(m) := \{M_{m,s} \in {\bf M}: s \in J^d(m)\}$ and ${\bf V}(m) := \text{span} {\bf M}(m)$. 
If $0 < p \le \infty,$ 
for all non-negative integers $m$ and all functions
\begin{equation} \label{StabIneq.1}
g = \sum_{s \in J^d(m)} a_s M_{m,s}
\end{equation}
from ${\bf V}(m)$, there is the norm equivalence 
\begin{equation} \label{StabIneq.2}
 \|g\|_p 
\ \asymp \ 2^{- dm/p}\|\{a_s\}\|_{p,m},
\end{equation}
where
\begin{equation*}
\| \{a_s\} \|_{p,m}
 \ := \
\biggl(\sum_{s \in J^d(m)}| a_s |^p \biggl)^{1/p} 
\end{equation*}
with the corresponding change when $p= \infty.$

For non-negative integer $k$, let the operator $q_k$ be defined by  
\begin{equation*}
q_k(f):= Q_k(f)- Q_{k-1}(f) \quad \text{with} \quad Q_{-1} (f) :=0.
\end{equation*}
From \eqref{eq:Reproducing^d} and \eqref{ConvergenceQ_m(f)} it is easy to see that 
a continuous function  $f$ has the decomposition
\begin{equation*} 
f \ =  \ \sum_{k =0}^\infty q_k(f)
\end{equation*}
with the convergence in the norm of $C({\II}^d)$.  
By using the B-spline refinement equation, one can represent 
the component functions $q_k(f)$ as
 \begin{equation} \label{q_k(1)}
q_k(f) 
= \ \sum_{s \in J^d(k)}c_{k,s}(f)M_{k,s},
\end{equation}
where $c_{k,s}$ are certain coefficient functionals of 
$f,$ which are defined as follows. For the univariate case, we put
\begin{equation}\label{def:c_{k,s}}
c_{k,s}(f)
\ := \
a_{k,s}(f) - a^\prime_{k,s}(f), \ k > 0,
\end{equation}
\begin{equation*}
a^\prime_{k,s}(f):= \ 2^{-2r+1} \sum_{(m,j) \in C(k,s)}
 \binom{2r}{j} a_{k-1,m}(f),  \ k > 0, \ \ 
a^\prime_{0,s}(f):= 0.
\end{equation*}
 and 
\begin{equation*} 
C(k,s) := \{(m,j): 2m + j - r = s, \ m \in J(k-1), \  0 \le j \le 2r\}, \ k > 0, \ \
C(0,s) := \{0\}.
\end{equation*} 
For the multivariate case, we define $c_{k,s}$ in the manner of the definition  
\eqref{a_{m,s}D} 
by
\begin{equation} \label{def:c_{k,s}M}
c_{k,s}(f) 
\ := \   
c_{k,s_1}((c_{k,s_2}(... c_{k,s_d}(f))).
\end{equation}

For functions $f$ on ${\II}^d$, we introduce the quasi-norms:
\begin{equation*} 
\begin{aligned}
B_2(f)
 \ & := \ 
\| \{ q_k(f) \}\|_{b_\theta^\alpha(L_p)}; \\
B_3(f)
\ & := \ 
\biggl(\sum_{k=0}^\infty 
\bigl( 2^{(\alpha - d/p)k}\|\{c_{k,s}(f)\}\|_{p,k} \bigl)^\theta \biggl)^{1/\theta}.
\end{aligned}
\end{equation*}
The following theorem has been proven in \cite{Di7}.

\begin{theorem} \label{Theorem:Representation}
Let $\ 0 < p, \theta \le \infty$ and $d/p < \alpha < 2r$. 
Then the hold the following assertions. 
\begin{itemize}
\item[(i)] A function $f \in B_{p, \theta}^\alpha$ can be represented by the mixed B-spline series 
\begin{equation} \label{eq:B-splineRepresentation}
f \ = \sum_{k = 0}^\infty \ q_k(f) = 
\sum_{k = 0}^\infty \sum_{s \in J^d(k)} c_{k,s}(f)M_{k,s}, 
\end{equation}  
satisfying 
the convergence condition
\begin{equation*} 
B_2(f) \ \asymp \ B_3(f) \ \ll \ B(f),
\end{equation*}
where the coefficient functionals $c_{k,s}(f)$ are explicitly constructed by formula 
\eqref{def:c_{k,s}}--\eqref{def:c_{k,s}M} as
linear combinations of at most $N$ function
values of $f$ for some $N \in \NN$ which is independent of $k,s$ and $f$.
 
\item[(ii)]  If in addition, $ \alpha < \min(2r, 2r - 1 + 1/p)$,
then a continuous function $f$ on ${\II}^d$ belongs to 
the Besov space $B_{p, \theta}^\alpha$ if and only if 
$f$ can be represented by the series \eqref{eq:B-splineRepresentation}.
Moreover, the Besov quasi-norm $B(f)$ is equivalent to one of the quasi-norms $B_2(f)$ and $B_3(f)$.
\end{itemize} 
\end{theorem}

%%%%%%%%%%%%%%%%%%%%%%%%%%%%%%%%%%%%%%%%%%%%%%%%%%%%%%%%%%%%%%%%%%%%%%%%%%%%%%%%%%%%%%%%%%%%%%%% 

\section{Adaptive continuous sampling recovery} 
\label{C_recovery}

In this section, we construct asymptotically optimal algorithms which give the upper bound 
of $\nu_n(B^\alpha_{p,\theta},L_q)$ in Theorem \ref{ThNU}. We need some auxiliary lemmas.

\begin{lemma} \label{Inequality[|f - Q_k(f)|_q]}
Let $ p,q , \theta, \alpha $  satisfy Condition \eqref{Condition} and $\alpha < 2r$. 
Then $Q_m \in \Cc(B_{p,\theta}^{\alpha}, L_q)$ and  for any $f \in B_{p,\theta}^{\alpha}$, we have 
\begin{equation*} %\label{ineq:[|Q_m(f)|_q]}
\|Q_m(f)\|_q 
\ \le \ 
C \|f\|_{B_{p,\theta}^{\alpha}},
\end{equation*}
\begin{equation} \label{ineq:[|f - Q_m(f)|_q]}
\|f - Q_m(f)\|_q 
\ \le \ 
C' 2^{- (\alpha - d(1/p - 1/q)_+ )m} \|f\|_{B_{p,\theta}^{\alpha}}
\end{equation}
with some constants $C, C'$ depending at most on $d, r, p, q , \theta$ and $\|\Lambda\|$.
\end{lemma}

\begin{proof} We first prove \eqref{ineq:[|f - Q_m(f)|_q]}.
The case when the Condition (ii) holds has been proven in \cite{Di7}. Let us prove 
the case when the Condition (i) takes place. 
We put $\alpha' := \alpha - d(1/p - 1/q)_+  > 0$. 
For an arbitrary $f \in B^\alpha_{p,\theta}$, by the representation \eqref{eq:B-splineRepresentation}
and \eqref{ineq:IneqSumf_k} we have
\begin{equation} \label{ineq:[|f - Q_m(f)|_q^tau]}
\|f - Q_m(f)\|_q^{\tau} 
\  \ll \ 
 \sum_{k > m} \| q_k(f)\|_q^{\tau}
\end{equation}
with any $\tau \le \min(q,1)$. From \eqref{q_k(1)} and \eqref{StabIneq.1}--\eqref{StabIneq.2}
we derive that
\begin{equation} \label{ineq:Nikolskii}
\|q_k(f)\|_q 
\  \ll \ 
 2^{(1/p - 1/q)_+k} \| q_k(f)\|_p
\end{equation}
Therefore, if $\theta \le \min(q,1)$, then
 we get 
\begin{equation*} 
\begin{aligned}
\|f - Q_m(f)\|_q  
\  & \ll \ 
 \left(\sum_{k > m} \| q_k(f)\|_q^{\theta} \right)^{1/\theta} \\
\  & \le \ 
 \left(\sum_{k > m} \{2^{(1/p - 1/q)_+k} \| q_k(f)\|_p\}^{\theta} \right)^{1/\theta} \\
\  & \le \ 
2^{- \alpha' m} \left(\sum_{k > m} \{2^{\alpha k}\| q_k(f)\|_p\}^{\theta} \right)^{1/\theta} \\
 \  & \ll \ 
2^{- \alpha' m} B_2(f)
 \ \ll \ 
2^{- \alpha' m} \|f\|_{B_{p,\theta}^{\alpha}}.
\end{aligned}
\end{equation*}
Further, if $\theta > \min(q,1)$, then from \eqref{ineq:[|f - Q_m(f)|_q^tau]} 
and \eqref{ineq:Nikolskii} it follows that
\begin{equation*}
\begin{aligned} 
\|f - Q_m(f)\|_q^{q^*}
\ & \ll \ 
\sum_{k > m} \| q_k(f)\|_q^{q^*} \\
 \  & \ll \ 
 \sum_{k > m} \{2^{\alpha k} \|q_k(f)\|_q \}^{q^*}\{2^{- \alpha' k}\}^{q^*},
\end{aligned}
\end{equation*}
where $q^* = \min(q,1)$. Putting $\nu := \theta/q^*$ and $\nu':= \nu/(\nu - 1)$, 
by H\"older's inequality obtain
\begin{equation*} 
\begin{aligned}
\|f - Q_m(f)\|_q^{q^*}
\ & \ll \ 
\left( \sum_{k > m} \{2^{\alpha k}\| q_k(f)\|_q\}^{q^* \nu}\right)^{1/\nu} 
\left(\sum_{k > m}\{2^{- \alpha' k}\}^{q^* \nu'}\right)^{1/\nu'} \\
\ & \ll \ 
\{B_2(f)\}^{q^*} \{2^{- \alpha' m}\}^{q^*}  
\  \ll \ 
\{2^{- \alpha' m}\}^{q^*}\|f\|_{B_{p,\theta}^{\alpha}}^{q^*}. 
\end{aligned}
\end{equation*}
Thus, the inequality \eqref{ineq:[|f - Q_m(f)|_q]} is completely proven.

Next, by use of the inequality
\begin{equation*}
\|Q_m(f)\|_q^{\tau} 
\  \ll \ 
 \sum_{k \le m} \| q_k(f)\|_q^{\tau}
\end{equation*}
with any $\tau \le \min(q,1)$, in a similar way we can prove that 
$\|Q_m(f)\|_q \ \ll \  \|f\|_{B_{p,\theta}^{\alpha}}$
and therefore, the inclusion $Q_m \in \Cc(B_{p,\theta}^{\alpha}, L_q)$.
\end{proof}

\begin{lemma} \label{Lemma:Q_m=L_n}
 For functions $f$ on ${\II}^d$, $Q_k$ defines a linear $n$-sampling algorithm 
of the form \eqref{L(f)}. More precisely, 
\begin{equation*} 
Q_k(f) 
\ = \ 
L_n(f) 
\ = \ 
\sum_{s \in I^d(k)} f(2^{-k}j) \psi_{k,j}, 
\end{equation*} 
where $n \ := (2^k + 1)^d$, \   
$\psi_{k,j}$ are explicitly constructed as linear combinations of at most 
$(2\mu + 2)^d$ B-splines $M_{k,s}$, and $I^d(k):=\{s \in \ZZ^d_+: 0 \le s_i \le 2^k, \ i =1,...,d\}$.
\end{lemma}

\begin{proof}
 For univariate functions the coefficient functionals 
$a_{k,s}(f)$ can be rewritten as 
\begin{equation*} 
a_{k,s}(f)
\  = \   
\sum_{|s - j| \le \mu} \lambda (s - j) f_k(2^{-k}j)
\  = \   
\sum_{j \in P(k,s)} \lambda_{k,s} (j) f(2^{-k}j),
\end{equation*}
where $\lambda_{k,s} (j) := \lambda(s - j)$ and
$P(k,s) = P_s(\mu) := \{j \in \{0,2^k\}: s-j \in P(\mu)\}$ for $\mu \le s \le 2^k - \mu$;  
$\lambda_{k,s} (j)$ is  a linear combination of no more than
$\max (2r, 2\mu +1)\le 2\mu + 2$ coefficients $\lambda (j), j \in P(\mu)$, for
$s < \mu$ or $s > 2^k - \mu$ and
\begin{equation*}
P(k,s)
\ \subset \
\begin{cases} 
P_s(\mu) \cup \{0,2r - 1\}, \ & s < \mu, \\
P_s(\mu)\cup \{2^k - 2r + 1, 2^k\}, \ & s > 2^k - \mu.
\end{cases}
\end{equation*}
If $j \in P(k,s)$, we have $|j - s|\le \max(2r, \mu) \le 2\mu + 2$. 
Therefore, 
$P(k,s) \subset P_s({\bar \mu})$, and we can rewrite the coefficient functionals $a_{k,s}(f)$
in the form 
\begin{equation*} 
a_{k,s}(f)
\  = \   
\sum_{j - s \in P(2\mu + 2)} \lambda_{k,s} (j) f(2^{-k}j)
\end{equation*}
with zero coefficients $\lambda_{k,s} (j)$ for $j \notin P(k,s)$. 
Therefore, for any $k \in {\ZZ}_+$, we have  
\begin{equation*}
\begin{aligned} 
Q_k(f)
\ & = \
\sum_{s \in J(k)} a_{k,s}(f) M^r_{k,s} 
\  = \   
\sum_{s \in J_r(k)} \sum_{j - s \in P(2\mu + 2)} \lambda_{k,s} (j) f(2^{-k}j) M^r_{k,s} \\
\ & = \   
\sum_{j \in I(k)} f(2^{-k}j) \sum_{s-j \in P(2\mu + 2)} \gamma_{k,j}(s)  M^r_{k,s}
\end{aligned}
\end{equation*}
for certain coefficients $\gamma_{k,j}(s)$. Thus,  the univariate $q_k(f)$ is of the form
\begin{equation*} 
Q_k(f)
\ = \
\sum_{j \in I(k)} f(2^{-k}j) \psi_{k,j},
\end{equation*}
where
\begin{equation*}
 \psi_{k,j}
\ := \
\sum_{s-j \in P(2\mu + 2)} \gamma_{k,j} (s)  M_{k,s},
\end{equation*}
are a linear combination of no more than the absolute number $2\mu + 2$ of B-splines $ M_{k,s}$, 
and the size $|I(k)|$ is $2^k$.
Hence, the multivariate $q_k(f)$ is of the form
\begin{equation*}
Q_k(f)
\ = \
\sum_{j \in I^d(k)} f(2^{-k}j) \psi_{k,j},
\end{equation*}
where
\begin{equation*}
 \psi_{k,j}
\ := \
\prod_{i=1}^d  \psi_{k,j_i}
\end{equation*}
are a linear combination of no more than the absolute number 
$(2\mu + 2)^d$ of B-splines $ M_{k,s}$, 
and the size $|I^d(k)|$ is $2^{dk}$. 
\end{proof}

For $0 < p \le \infty$, denote by $\ell_p^m$ the space of all sequences
$x=\{x_k\}_{k=1}^m$ of numbers, equipped with the quasi-norm
\begin{equation*}
\|x\|_{\ell^m_p} 
\ := \ \left( \sum_{k=1}^m |x_k|^p\right)^{1/p}
\end{equation*}
with the change to the $\max$ norm when $p=\infty$. Denote by 
$B^m_p$ the unit ball in $\ell^m_p$. Let 
${\mathcal E} = \{e_k\}_{k=1}^m $ be the canonical basis in $\ell_q^m$, i. e., 
$x=\sum_{k=1}^m x_k e_k.$

For $x=\{x_k\}_{k=1}^m \in \ell_q^m$, we let the set 
$\{k_j\}_{j=1}^m$ be ordered  so that
\begin{equation*}
|x_{j_1}| \ge |x_{j_2}| \ge \cdots |x_{j_s}| \ge \cdots \ge |x_{j_m}|. 
\end{equation*}

We define the algorithm  $P_n$ for the $n$-term approximation with regard to 
the basis ${\mathcal E}$ in the space 
$\ell^m_q \  (n\le m)$ as follows. For $x=\{x_k\}_{k=1}^m \in \ell_p^m$, we let the set 
$\{k_j\}_{j=1}^m$ be ordered  so that
\begin{equation*}
|x_{j_1}| \ge |x_{j_2}| \ge \cdots |x_{j_s}| \ge \cdots \ge |x_{j_m}|. 
\end{equation*}
Then, for $n < m$ we define
\begin{equation*}
P_n(x) :=\sum_{j=1}^n (x_{k_j} - |x_{n+1}| \, \text{sign} \, x_{k_j})e_{k_j}.
\end{equation*}
For a proof of the following lemma see \cite{Di4}.

\begin{lemma} \label{C_Algorithm}
 The operator $P_n \in \Cc(\ell_p^m, l_q^m)$ for $0 < p, q \le \infty$. 
If $0 < p < q \le \infty$, then we have for any positive integer $n<m$
\begin{equation*}
\sup_{x \in B_p^m} \|x - P_n(x)\|_{\ell_q^m}
 \  \le \
n^{1/q -1/p}.
\end{equation*}
\end{lemma}

The following theorem gives the upper bound of \eqref{Asymp:[nu]} in Theorem \ref{ThNU}.

\begin{theorem} \label{ThUpperBound} 
Let $ p,q , \theta, \alpha $  satisfy Condition \eqref{Condition}. 
Then for the $d$-variable Besov space 
$B^\alpha_{p,\theta} $,  there is 
the following upper bound
\begin{equation} \label{asympN}
\nu_n(B^\alpha_{p,\theta},L_q) 
\ \ll \ 
n^{- \alpha / d}. 
\end{equation}
If in addition, $\alpha < 2r$, we can find an positive integer $k^*$ and a continuous 
$n$-sampling recovery algorithm $S_n^B \in \Cc(B_{p,\theta}^{\alpha}, L_q)$
of the form \eqref{S_n^B(f)} with $A = \Sigma_n({\bf M}(k^*))$, such that
\begin{equation} \label{UpperBound_N}
\sup_{f \in SB^\alpha_{p,\theta}} \, \|f - S_n^B(f)\|_q 
\ll 
n^{- \alpha / d}.
\end{equation}
\end{theorem}

\begin{proof} 
We will prove \eqref{UpperBound_N} 
and therefore, \eqref{asympN}.  We first consider the case $p \ge q$. For any integer 
$k^*$, by Lemmas \ref{Lemma:Q_m=L_n} and \ref{Inequality[|f - Q_k(f)|_q]} we have 
\begin{equation} \label{Q_mError(3)}
\sup_{f \in SB^\alpha_{p,\theta}} \, \|f - Q_{k^*}(f)\|_q 
 \ \asymp \ 
2^{-\alpha k^*}.
\end{equation}
The number of sampled values of $f$ in $Q_{k^*}(f)$ is $n^*:= (2^{k^*} + 1)^d.$ 
For a given integer $n$ (not smaller than $2^d$), define $k^*$ by the condition 
\begin{equation} \label{ChoiceN}
 C n \le n^* = (2^{k^*} + 1)^d \le n,
\end{equation}
 with $C$ an absolute constant.
By Lemma \ref{Inequality[|f - Q_k(f)|_q]} $Q_{k^*} \in \Cc(B_{p,\theta}^{\alpha}, L_q)$.
By the choice of $k^*,$ $Q_{k^*}(f)= S_n^B(f)$ is a 
linear $n$-sampling algorithm $S_n^B(f)$ is of the form \eqref{L(f)} with 
$A = \Sigma_n({\bf M}(k^*))$ and ${\bf M}(k^*)\in {\mathcal G}$ as a finite family. 
Therefore, by \eqref{Q_mError(3)} 
and \eqref{ChoiceN} we receive \eqref{UpperBound_R} for the case $p \ge q$. 

We next treat the case $p < q$. For arbitrary positive integer 
$m,$ a function $f \in SB_{p,\theta}^{\alpha}$  
can be represented by a series
\begin{equation} \label{q_kRepresentation(2)}
f \ =  \ \sum_{k = 0}^m \sum_{s \in J^d(k)}a_{k,s}(f)M_{k,s}
 + \sum_{k = m + 1}^\infty \sum_{s \in J^d(k)}c_{k,s}(f)M_{k,s}
\end{equation}
converging in the norm of $B^\alpha_{p,\theta}$ or, equivalently,
\begin{equation} \label{q_kRepresentation(2b)}
f \ =  \ Q_{m}(f) + \sum_{k = m + 1}^\infty q_k(f)
\end{equation}
with the component functions
\begin{equation} \label{q_k(2)}
q_k(f)
=  \ \sum_{s \in J^d(k)}c_{k,s}(f)M_{k,s}
\end{equation}
from the subspace ${\bf V}(k)$.
Moreover, $q_k(f)$ satisfy the condition
\begin{equation} \label{q_k(f)}
 \|q_k(f)\|_p 
\ \asymp \ 2^{- dk/p}\|\{c_{k,s}(f)\}\|_{p,k}
\ \ll \   2^{-\alpha k}, \quad k = m + 1,m + 2,... 
\end{equation}
The representation  
\eqref{q_kRepresentation(2)}--\eqref{q_k(f)} follows from Theorem \ref{Theorem:Representation} 
for the case (i) in Condition \eqref{Condition}, 
and from Lemma  \ref{Inequality[|f - Q_k(f)|_q]} for the case (ii) in Condition \eqref{Condition}. 

Put $m(k):=|J^d(k)| =  (2^k + 2r - 1)^d.$ 
Let $\bar{k}, k^*$ be non-negative integers with 
$\bar{k} <  k^*,$ and $\{ n(k) \}_{k = \bar{k} + 1}^{k^*}$ a sequence of non-negative integers with 
$n(k) \le m(k)$. We will construct a recovering function of the form
\begin{equation} \label{G(f)(1)}
G(f) := \ \sum_{s \in J(\bar{k})} a_{k,s}(f) M_{k,s} \ 
+ \ \sum_{k=\bar{k}+1}^{k^*} \sum_{j=1}^{n(k)} c_{k,s_j}(f) M_{k,s_j},
\end{equation}
with $s_{k,j} \in J^d(k)$, or equivalently,
\begin{equation} \label{G(f)(2)}
G(f)= \ Q_{\bar{k}}(f) \ + \  \sum_{k=\bar{k}+1}^{k^*} G_k(f).
\end{equation}
 
The algorithms $G_k$ are constructed as follows. For a $f \in SB_{p,\theta}^{\alpha}$, we take the
sequence of coefficients $\{c_{k,s}(f)\}_{s \in J^d(k)}$ and reorder the indexes 
$s \in J^d(k)$  as $\{s_j\}_{j=1}^{m(k)}$
so that
\begin{equation*} 
|c_{k,s_1}(f)| \ge |c_{k,s_2}(f)| 
\ge \cdots |c_{k,s_n}(f)| \ge \cdots |c_{k,m(k)}(f)|, 
\end{equation*}
and then define
\begin{equation*} 
G_k(f) 
\ := \
\sum_{j=1}^{n(k)} \{c_{k,s_j}(f) - |c_{k,s_{n(k) + 1}}(f)|\, \text{sign} \, c_{k,s_j}(f)\} M_{k,s_j}.
\end{equation*}

We prove that $G \in \Cc(B_{p,\theta}^{\alpha}, L_q)$. For $0 < \tau \le \infty$, denote by ${\bf V}(k)_\tau$ 
the quasi-normed space of all functions $f \in {\bf V}(k),$ equipped 
with the quasi-norrm $L_\tau.$ Then by Lemma \ref{Inequality[|f - Q_k(f)|_q]}
$q_k \in \Cc(B_{p,\theta}^{\alpha}, {\bf V}(k)_p)$.
Consider the sequence $\{c_{k,s}(f)\}_{s \in J^d(k)}$ as an element in $\ell_p^{m(k)}$ 
and let the operator $D_k: {\bf V}(k)_p \to \ell_p^{m(k)}$ be defined by $g \mapsto \{a_s\}_{s \in J^d(k)}$ if
$g \in {\bf V}(k)_q$ and $g  = \sum_{s \in J^d(k)} a_s M_{k,s}$. 
Obviously, by \eqref{StabIneq.1}--\eqref{StabIneq.2} $D_k \in \Cc(\Sigma(k)_p, \ell_p^{m(k)})$.
For $x=\{x_{k,s}\}_{s \in J^d(k)} \in l_p^{m(k)}$, we let the set 
$\{k_j\}_{j=1}^{m(k)}$ be ordered  so that
\begin{equation*}
|x_{j_1}| \ge |x_{j_2}| \ge \cdots |x_{j_s}| \ge \cdots \ge |x_{j_{m(k)}}| 
\end{equation*}
and define
\begin{equation*}
P_{n(k)}(x) :=\sum_{j=1}^{n(k)} (x_{k_j} - |x_{n(k)+1}| \, \text{sign} \, x_{k_j})e_{k_j}.
\end{equation*}
Temporarily denote by $H$ the quasi-metric space  of all 
$x=\{x_{k,s}\}_{s \in J^d(k)} \in \ell_q^{m(k)}$ for which $ x_k= 0, k \notin Q$, 
for some subset $Q \subset J^d(k)$ with $|Q|=n(k)$. The quasi-metric of $H$ is generated by 
the quasi-norm of $\ell_q^{m(k)}$. By Lemma \eqref{C_Algorithm} we have $P_{n(k)} \in \Cc(\ell_p^m, H)$.
Consider the mapping $R_{{\bf M}(k)}$ from $H$ into 
$\Sigma_{n(k)}({\bf M}(k))$ defined by
\begin{equation*}
R_{{\bf M}(k)}(x):= \  \sum_{s \in Q} x_{k,s} M_{k,s},
\end{equation*}
if $x=\{x_{k,s}\}_{s \in J^d(k)} \in H$ and $ x_k = 0, k \notin Q$, for some $Q$ with
$|Q|=n(k)$. 
Since the family {\bf M}(k) is bounded in $L_q$, it is easy to verify that 
$R_{{\bf M}(k)} \in \Cc(H,L_q)$. 
We have
\begin{equation*} 
G_k 
\ = \
R_{{\bf M}(k)}\circ P_{n(k)} \circ D_k \circ q_k.
\end{equation*}
Hence, $G_k \in \Cc(B_{p,\theta}^{\alpha}, L_q)$ as  the supercomposition of continuous operators.
Since by Lemma \ref{Inequality[|f - Q_k(f)|_q]} $Q_{\bar{k}}(f) \in \Cc(B_{p,\theta}^{\alpha}, L_q)$, 
from \eqref{G(f)(2)} it follows $G \in \Cc(B_{p,\theta}^{\alpha}, L_q)$. 

Notice that in the operator $G$, the quasi-interpolant $Q_{\bar{k}}(f)$ 
is the main non-adaptive linear part. 
Its adaptive non-linear part is 
a sum of continuous algorithms $G_k$ 
 for a continuous adaptive approximation of each component 
function $q_k(f)$ in the $k$th scale subspaces 
${\bf V}(k)$, $\bar{k} < k \le k^*.$

Let $m$ be the number of the terms in the sum \eqref{G(f)(1)}. 
Then, $G(f) \in \Sigma_m({\bf M}(k^*))$ and 
\begin{equation*}
m 
\ = \
(2^{ \bar{k}} + r -1)^d +  \sum_{k=\bar{k} + 1}^{k^*} n(k). 
\end{equation*}
Moreover, the number of sampled values defining $G(f)$ does not exceed 
\begin{equation*} 
m' := (2^{ \bar{k}} + 1)^d + (2 \mu + 2r)^d \sum_{k=\bar{k} + 1}^{k^*} n(k) . 
\end{equation*}

Let us select $\bar{k},  k^*$ and a sequence $\{ n(k) \}_{k = \bar{k} + 1}^{k^*}$.
 We define an integer 
$\bar{k}$ from the condition
\begin{equation} \label{bar{k}}
C_1 2^{d \bar{k}} \le n < C_2 2^{d \bar{k}},
\end{equation} 
where $C_1, C_2$ are absolute constants which will be chosen  below.

Notice that under the hypotheses of Theorem 
\ref{ThNU} we have $0 < \delta  < \alpha.$ Further, we fix a number 
$\varepsilon$ satisfying the inequalities
\begin{equation} \label{varepsilon}
0 < \varepsilon < (\alpha - \delta ) / \delta,
\end{equation}
where $\delta := d(1/p - 1/q)$. An appropriate selection of $k^*$ and 
$\{n(k)\}_{k=\bar{k}+1}^{k^*}$ is 
\begin{equation} \label{k^*}
k^*:= \ [\varepsilon^{-1} \log (\lambda n)] + \bar{k} +1.
\end{equation}
and 
\begin{equation} \label{n(k)}
n(k) \ = \ 
[\lambda n2^{-\varepsilon(k-\bar{k})}], \quad  
k = \bar{k} +1, \bar{k} + 2, ..., k^*, 
\end{equation} 
with a positive constant $\lambda.$ Here $[a]$ denotes the integer part of the number $a$. 
It is easy to find constants $C_1, C_2$ 
in \eqref{bar{k}} and $\lambda$ in \eqref{n(k)} so that $n(k) \le m(k), k = \bar{k} + 1,..., k^*,$ 
$m \le n$ and $m' \le n.$ Therefore, $G$ is an $n$-sampling algorithm 
$S_n^B$ of the form \eqref{S_n^B(f)} with $B = \Sigma_m({\bf M}(k^*))$  and
${\bf M}(k^*)\in {\mathcal G}$ as a finite family.
Let us give a upper bound for $\|f - S_n^B(f)\|_q$. For a fixed number $0 < \tau \le \min (p,1)$, we have
by  \eqref{ineq:IneqSumf_k},  
\begin{equation}\label{ineq:|f - S_n^B(f)|_q}
\|f - S_n^B(f)\|_q^\tau 
\  \le \ 
\sum_{k=\bar{k}+1}^{k^*} \|q_k(f) -  G_k(q_k(f))\|_q^\tau 
\ + \ \sum_{k > k^*} \|q_k(f) \|_q^\tau.   
\end{equation}
By \eqref{StabIneq.1}--\eqref{StabIneq.2} and \eqref{q_k(f)}
we have for all $f \in SB_{p,\theta}^{\alpha}$ 
\begin{equation} \label{q_k(f)2}
 \|q_k(f)\|_q 
\ \ll \   2^{-(\alpha - \delta) k}, \quad k = k^* + 1, + 2,... 
\end{equation}
Further, we will estimate $ \|q_k(f) -  G_k(q_k(f))\|_q$ for 
all $f \in SB_{p,\theta}^{\alpha}$ and $k = \bar{k} + 1,..., k^*$. 
From Lemma \ref{C_Algorithm} we get
\begin{equation} \label{ErrorG_n}
\left( \sum_{j=n(k) + 1}^{m(k)} |c_{k,s_j}(f)|^q \right)^{1/q}
\  \le  \  \{n(k)\}^{- \delta}\|\{c_{k,s}(f)\}\|_{p,k}.
\end{equation}
By \eqref{StabIneq.1}--\eqref{StabIneq.2}, \eqref{q_k(f)2} and \eqref{ErrorG_n} we obtain for 
all $f \in SB_{p,\theta}^{\alpha}$ and $k = \bar{k} + 1,..., k^*$ 
\begin{equation} \label{ineq:Error[q_k -  G_k]}
\begin{aligned}
\left\|q_k(f) -  G_k(q_k)\right\|_q 
& = \left\|\sum_{j=n(k) +1}^{m(k)} c_{k,s_j}(f) M_{k,s_j}\right\|_q  
 \asymp \ 2^{-k/q}\left( \sum_{j=n(k) + 1}^{m(k)} |c_{k,s_j}(f)|^q \right) ^{1/q} \\
& \ll  \ 2^{ - k/q}\{n(k)\}^{- \delta }\|\{c_{k,s}(f)\}\|_{p,k} 
 \ll  \ 2^{-\alpha k}2^{\delta k}\{n(k)\}^{- \delta}.
\end{aligned}
\end{equation}
From \eqref{ineq:|f - S_n^B(f)|_q} by using  \eqref{ineq:Error[q_k -  G_k]}, \eqref{q_k(f)2},
\eqref{bar{k}}--\eqref{n(k)} 
and the inequality $\alpha  > \delta,$ 
we derive that for all functions $f \in SB_{p,\theta}^{\alpha}$
\begin{equation*}
\begin{aligned}
\|f - S_n^B(f)\|_q^\tau 
\ & \ll 
\sum_{k=\bar{k}+1}^{k^*}2^{- \tau  \alpha k}2^{\tau \delta k}\{n(k)\}^{- \tau \delta}
\ + \ \sum_{k = k^* +1}^\infty 2^{- \tau \alpha k}2^{\tau \delta k} \\
\ & \ll 
n^{- \tau \delta} 2^{- \tau (\alpha - \delta)\bar{k}}
\sum_{k=\bar{k}+1}^{k^*}2^{- \tau (\alpha - \delta + \delta \varepsilon)(k - \bar{k})}
\ + \ 2^{- \tau (\alpha - \delta)k^*}
\sum_{k = k^* +1}^\infty 2^{- \tau(\alpha - \delta)(k - k^*)}  \\
\ & \ll 
n^{- \tau \delta} 2^{- \tau (\alpha - \delta)\bar{k}}
\ + \ 2^{- \tau (\alpha - \delta)k^*} 
\  \ll \ n^{- \tau \alpha /d}. 
\end{aligned}
\end{equation*}
Summing up, we have proven that the constructed  $n$-sampling algorithm 
$G = S_n^B(f) \in \Cc(B_{p,\theta}^{\alpha}, L_q)$ and is of the form \eqref{S_n^B(f)} with 
$A= \Sigma_m({\bf M}(k^*))$,  and ${\bf M}(k^*)\in {\mathcal G}$ as a finite family
for which  the inequality  \eqref{UpperBound_R}
holds true for the case $p < q$.
\end{proof}

\section{Lower bounds of $\nu_n(B^\alpha_{p,\theta},L_q)$ } 
\label{Lowerbounds}

To prove the lower bound  Theorem \ref{ThNU} we compare
$\nu_n(B^\alpha_{p,\theta},L_q)$ with a related non-linear $n$-width which is defined on the basis
of continuous algorithms in $n$-term approximation.

 Let $X, Y$ be  quasi-normed spaces and $X$ is a linear subspace of $Y$. 
Let $W$ be a subset in $X$ and $\Phi = \{\varphi_k\}_{k \in K}$ a family of elements in $Y$. 
Denote by ${\mathcal G}(Y)$ the set of all bounded families $\Phi \subset Y$
whose intersection $\Phi \cap L$ with any finite dimensional subspace $L$
in $Y$ is a finite set. We define the non-linear $n$-width $\tau_n^X(W,Y)$ by
\begin{equation*}
\tau_n^X(W,Y)
\  := \ 
\inf_{\Phi \in {\mathcal G}(Y)}  \  \inf_{S \in \Cc(X,Y): \ S(X) \subset \Sigma_n(\Phi)}  \ 
\sup_{f \in W} \ \| f - S(f) \|_Y.
\end{equation*}
Since all quasi-norms in a finite dimensional linear space are equivalent, we will
drop $"X"$ in the notation $\tau_n^X(W,Y)$ for the case where $Y$ is finite dimensional.

Denote by $SX$  the 
unit ball in the quasi-normed space $X$. By definition we have
\begin{equation} \label{ineq:[nu>tau^B]}
\nu_n(B^\alpha_{p,\theta},L_q) 
\ \ge \ 
\tau_n^B(SB^\alpha_{p,\theta},L_q), 
\end{equation}
where we use the abbreviation: $B := B^\alpha_{p,\theta}$.
 
\begin{lemma} \label{[tau_{n+m}(W,Y)<]}
Let the linear space $L$ be equipped with two equivalent quasi-norms $\|\cdot\|_X$ and 
$\|\cdot\|_Y$, $W$ a subset of $L$. 
If  $\tau_n^X(W,Y)>0$, we have
\begin{equation*}
\tau_{n+m}^X(W,X) \ \le \ \tau_n^X(W,Y) \, \tau_m^X(SY,X).
\end{equation*}
\end{lemma}
\begin{proof}
This lemma can be proven is a way similar to the proof of Lemma 4 in {\cite{Di3}}.
\end{proof}

\begin{lemma} \label{tau>}
Let $0 <  q \le \infty$. Then we have  for any positive integer $n < m$
\begin{equation*}
\tau_n(B^m_\infty,\ell^m_q) \ \ge \  \frac{1}{2}(m -n - 1)^{1/q}.
\end{equation*}
\end{lemma}
\begin{proof}
 We need the following inequality.
If $W$ is a compact subset in the
finite dimensional normed space $Y$, then we have \cite{Di3}
\begin{equation} \label{[tau_n(W,X)>]}
2\tau_n(W,Y) \ \ge \ b_n(W,Y),
\end{equation}
where the Bernstein $n$-width $b_n(W,Y)$ is defined by
\begin{equation*}
b_n(W,Y):= \   \sup_ {L_{n+1}} \, \sup \{t >0: \ tSY \cap L_{n+1} \ \subset \ W \}
\end{equation*}
with the outer supremum taken over all $(n+1)$-dimensional linear
manifolds $L_{n+1}$ in $Y$.

 By definition we have
\begin{equation*}
b_{m-1}(B^m_\infty,\ell^m_\infty) \ = \  1.
\end{equation*}
Hence, by  \eqref{[tau_n(W,X)>]}, Lemmas \ref{C_Algorithm} and  \ref{[tau_{n+m}(W,Y)<]}
we derive that
\begin{equation*}
\begin{aligned}
1
 \ = \  
b_{m-1}(B^m_\infty,\ell^m_\infty)
 \ & \le \ 
2 \tau_{m-1}(B^m_\infty,\ell^m_\infty) \\
 \ \le \ 
2 \tau_n(B^m_\infty,\ell^m_q)\tau_{m-n-1}(B^m_q,\ell^m_\infty)
 \ & \le \ 
 2(m - n - 1)^{-1/q}\tau_n(B^m_\infty,\ell^m_q).
 \end{aligned}
\end{equation*}
This proves the lemma.
\end{proof}

\begin{theorem} \label{[Thnu>>]} 
Let $0 < p,q, \theta \le \infty$ and $\alpha > 0$. 
Then we have
\begin{equation*} 
\nu_n(B^\alpha_{p,\theta},L_q)  
\ \gg \ 
n^{- \alpha / d}. 
\end{equation*}
\end{theorem}

\begin{proof} 
By \eqref{ineq:[nu>tau^B]} the theorem follows from the 
inequality
\begin{equation} \label{ineq:tau^B>>}
\tau_n^B(SB^\alpha_{p,\theta},L_q)  
\ \gg \ 
n^{- \alpha / d}. 
\end{equation}
To prove \eqref{ineq:tau^B>>} we will need an additional inequality. 
Let $Z$ is a subspace of the quasi-normed  space
$Y$ and $W$ a subset of the quasi-normed  space
$X$. If $P:Y \to Z$ is a linear
projection such that $\|P(f)\|_Y \le \lambda\|f\|_Y (\lambda >0)$ for every $ f \in Y$, 
then it is easy to verify that
\begin{equation} \label{ineq:[tau_n^Y(W,X)]}
\tau_n^X(W,Y) \ \ge \ \lambda^{-1}\tau_n^X(W,Z).
\end{equation}

Because of the inclusion 
$ U:= SB^\alpha_{\infty,\theta} \subset SB^\alpha_{p,\theta}$, we have
\begin{equation} \label{ineq:[tau_n^B]}
\tau_n^B(SB^\alpha_{p,\theta},L_q)   
\ \ge \ 
\tau_n^B(U,L_q).
\end{equation}
Fix an integer $r$ with the condition
$\alpha < \min (2r, 2r - 1 +1/p, 2r)$. 
Let $U(k):= \{ f \in {\bf V}(k):  \|f\|_\infty \le 1\}.$  
For each $f \in {\bf V}(k),$ there holds the Bernstein inequality \cite{DP}.
\begin{equation*} 
\| f \|_{B^\alpha_{\infty,\theta}} 
\ \le \ 
C 2^{\alpha k} \| f \|_\infty, 
\end{equation*}
where $C > 0$ does not depend on $f$ and $k$.
Hence,
$C^{-1} 2^{-\alpha k} U(k)$ is a subset in $U$.
This implies the inequality
\begin{equation} \label{4.1}
\tau_n^B(U,L_q)  \  
\ \gg \ 2^{- \alpha k}
\tau_n^B(U(k),L_q).
\end{equation}
Denote by ${\bf V}(k)_q$ the quasi-normed space of all functions $f \in {\bf V}(k),$ equipped 
with the quasi-norm $L_q.$ Let $T_k$ be the bounded linear projector from $L_q$ onto ${\bf V}(k)_q$ 
constructed in \cite{DP} such that
$\|T_k(f)\|_q \le \lambda'\|f\|_q$ for every 
$f \in L_q$ , where $\lambda'$ is an absolute constant.  
Therefore, by \eqref{ineq:[tau_n^Y(W,X)]}
\begin{equation} \label{4.2}
\tau_n^B(U(k),L_q) \  
\ \gg \ \tau_n^B(U(k),{\bf V}(k)_q)
\ = \ \tau_n(U(k),{\bf V}(k)_q).
\end{equation}
Observe that 
$m:= \ |J^d(k)| \ = \ \text{dim} {\bf V}(k)_q = (2^k + 2r - 1)^d \ \asymp \ 2^{dk}.$ 
For a non-negative integer $n$, define $m=m(n)$ from  the condition 
\begin{equation} \label{4.4}
n \ \asymp \ 2^{dk} \  \asymp \ m \ > \ 2n.
\end{equation}
Consider the quasi-normed space $\ell^m_q$ of all sequences $\{a_s\}_{s \in J^d(k)}$.
Let the natural continuous linear one-to-one mapping $\Pi$ from ${\bf V}(k)_q$ onto $\ell^m_q$ be defined  by
\begin{equation*}
\Pi (f):= \ \{a_s\}_{s \in J^d(k)}
\end{equation*}
if $f \in {\bf V}(k)_q$ and $f  = \sum_{s \in J^d(k)} a_s M_{k,s}$.
We have by \eqref{StabIneq.1}--\eqref{StabIneq.2} 
$\|f\|_\infty  \asymp \|\Pi (f)\|_{\ell^m_\infty}$ and 
$\|f\|_q  \asymp 2^{-dk/q}\|\Pi (f)\|_{\ell^m_q}.$
Hence, we obtain by Lemma \ref{tau>}
\begin{equation*}
\begin{aligned}
\tau_n(U(k), {\bf V}(k)_q)
\ & \asymp \ 
2^{- dk/q}
\tau_n (B^m_\infty, \ell^m_q) \\ 
\ & \gg \ 2^{- dk/q} (m - n - 1)^{1/q}
\  \gg \ 1.
\end{aligned}
\end{equation*}
Combining the last estimates and \eqref{ineq:[tau_n^B]}--\eqref{4.4} 
completes the proof of \eqref{ineq:tau^B>>}.
\end{proof}
%%%%%%%%%%%%%%%%%%%%%%%%%%%%%%%%%%%%%%%%%%%%%%%%%%%%%%%%%%%%%%%%%%%%%%%%%%%%%%%%%%%%%%%%%%

\section{Adaptive non-continuous sampling recovery} 
\label{NC_recovery}
In this section, we prove the asymptotic order of  $s_n(B^\alpha_{p,\theta},{\bf M},L_q)$, 
$r_n(SB^\alpha_{p,\theta})_q$ and $e_n(SB^\alpha_{p,\theta})_q$ 
in Theorem \ref{ThNU}.

Let $W$ and~$B$ be subsets in $L_q$. 
For approximation of elements from $W$ by $B$, the quantity
\begin{equation*}
E(W, B)_q := \  \sup_{f \in W} \inf_{\varphi \in  B} \|f-\varphi\|_q
\end{equation*}
gives the worst case error of approximation. 

Let $\Phi = \{\varphi_k\}_{k \in K}$ 
be a family of elements in $L_q$. 
The  quantity of $n$-term approximation 
$\sigma_n(W,\Phi)_q$  
with regard to $\Phi,$ is defined by
\begin{equation*}
\sigma_n(W,\Phi)_q
\ := \ E(W, \Sigma_n(\Phi))_q.
\end{equation*}

Given a family ${\mathcal B}$ of subsets in $L_q$, we can consider the best 
approximation by $B$ from ${\mathcal B}$ in terms of the quantity
\begin{equation} \label{d(W,B,X)}
d(W, {\mathcal B})_q := \  \inf_{B \in {\mathcal B}}E(W, B)_q.
\end{equation}

Notice the following useful identities
\begin{equation*}
\sigma_n(W,\Phi)_q
\ = \ \inf_{S:\, W \to \Sigma_n(\Phi)} \ \sup_{f \in W}\ \|f - S(f)\|_q.
\end{equation*}
and
\begin{equation} \label{eq:d(W,B,X)}
d(W, {\mathcal B})_q
\ = \ \inf_{B\in {\mathcal B}} \ \inf_{S^B:\, W \to B} 
\ \sup_{f \in W}\ \|f - S^B(f)\|_q.
\end{equation}

The quantity $d(W, {\mathcal B})_q$ is called the entropy $n$-width (entropy number)
$\varepsilon_n(W)_q$ if ${\mathcal B}$ in \eqref{d(W,B,X)} is the family of all subsets $B$
of $L_q$ such that $|B| \le 2^n$. 
The non-linear $n$-width $\rho_n(W)_q$ is defined only when $L_q$ is a space of real-valued
functions on a set $\Omega$, if ${\mathcal B}$ in \eqref{d(W,B,X)}
is the family of all subsets in $L_q$ of pseudo-dimension at most $n.$ 

From \eqref{eq:d(W,B,X)}
we have
\begin{equation*}
\varepsilon_n(W)_q
\ = \ \inf_{|B| \le 2^n} \ \inf_{S^B:\, W \to B} \ \sup_{f \in W}\ \|f - S^B(f)\|_q,
\end{equation*}
and
\begin{equation*} 
\rho_n(W)_q
\ = \ \inf_{\dimp B \le n} \ \inf_{S^B:\, W \to B} \
\sup_{f \in W}\ \|f - S^B(f)\|_q.
\end{equation*}
Therefore, we can take the last identities as alternative definitions 
of $\varepsilon_n(W)_q$ and $\rho_n(W)_q$.

\begin{theorem} 
Let $ p,q , \theta, \alpha $  satisfy Condition \eqref{Condition} and $\alpha < 2r$. 
Then for the $d$-variable Besov class 
$SB^\alpha_{p,\theta}$,  we can explicitly construct an $n$-sampling algorithm $S_n^B$ with
 $B = \Sigma_n({\bf M})$ so that
\begin{equation} \label{asympS}
\sup_{f \in SB^\alpha_{p,\theta}} \|f - S_n^B(f)\|_q
\ \asymp \ 
s_n(B^\alpha_{p,\theta}, {\bf M}, L_q)  
\ \asymp \
n^{- \alpha / d}.
\end{equation}
\end{theorem}

\begin{proof}
In \cite[Corollary2.3, Theorem 3.2]{Di7} an $n$-sampling algorithm $S_n^B$ with
 $B = \Sigma_n({\bf M})$ was explicitly constructed such that
\begin{equation*}
\sup_{f \in SB^\alpha_{p,\theta}} \|f - S_n^B(f)\|_q
\ \ll \ 
n^{- \alpha / d}.
\end{equation*}
 This proves the upper bound of \eqref{asympS}. 

The lower bound 
follows from the inequality 
$s_n(B^\alpha_{p,\theta}, {\bf M}, L_q) \ \ge \ \sigma_n(SB^\alpha_{p,\theta}, {\bf M})_q $ and 
the inequality 
\begin{equation*} 
\sigma_n(SB^\alpha_{p,\theta}, {\bf M})_q  
\ \gg \
n^{- \alpha / d}.
\end{equation*}
which was proven in \cite[Theorem 5.1]{Di7} 
\end{proof}
\begin{theorem} 
Let $ p,q , \theta, \alpha $  satisfy Condition \eqref{Condition}. 
Then for the $d$-variable Besov class 
$SB^\alpha_{p,\theta}$,  there is 
the following asymptotic order
\begin{equation} \label{asympR}
r_n(SB^\alpha_{p,\theta})_q 
\ \asymp \ 
n^{- \alpha / d}. 
\end{equation}
If in addition, $\alpha < 2r$, we can explicitly construct a subset  $B$ in
$\Sigma_n({\bf M})$ having  $\dimp(B) \le n$, and a  
sampling recovery method $S_n^B$
of the form \eqref{S_n^B(f)}, such that
\begin{equation} \label{UpperBound_R}
\sup_{f \in SB^\alpha_{p,\theta}} \|f - S_n^B(f)\|_q 
\ll 
n^{- \alpha / d}.
\end{equation}
\end{theorem}

\begin{proof} 
The inequality \eqref{UpperBound_R} and therefore, the upper bound of \eqref{asympR}
was proven in \cite[Theorem 3.1]{Di7}. 

The lower bound 
follows from the inequality 
$r_n(SB^\alpha_{p,\theta})_q \ \ge \ \rho_n(SB^\alpha_{p,\theta})_q$ and 
the inequality 
\begin{equation*} 
 \rho_n(SB^\alpha_{p,\theta})_q  
\ \gg \
n^{- \alpha / d}.
\end{equation*}
which was proven in \cite[Theorem 5.3]{Di7}
\end{proof}

\begin{theorem} 
Let $ p,q , \theta, \alpha $  satisfy Condition \eqref{Condition}. 
Then for the $d$-variable Besov class 
$SB^\alpha_{p,\theta}$,  there is 
the following asymptotic order
\begin{equation} \label{asympE}
e_n(SB^\alpha_{p,\theta})_q 
\ \asymp \ 
n^{- \alpha / d} . 
\end{equation}
If in addition, $\alpha < 2r$, we can explicitly construct  a subset  $B$ in
$\Sigma_n({\bf M})$ having $|B| \le 2^n$, and a 
sampling recovery method $S_n^B$
of the form \eqref{S_n^B(f)}, such that
\begin{equation} \label{UpperBndE}
\sup_{f \in SB^\alpha_{p,\theta}} \|f - S_n^B(f)\|_q 
\ll 
n^{- \alpha / d}.
\end{equation}
\end{theorem}

\begin{proof} 
The inequality \eqref{UpperBndE} and therefore, the upper bound of \eqref{asympE}
was proven in \cite[Theorem 4.1]{Di7}. 

The lower bound 
follows from the inequality 
$e_n(SB^\alpha_{p,\theta})_q \ \ge \ \varepsilon_n(SB^\alpha_{p,\theta})_q$ and 
the inequality 
\begin{equation*} 
 \varepsilon_n(SB^\alpha_{p,\theta})_q  
\ \gg \
n^{- \alpha / d}.
\end{equation*}
which was proven in \cite[Theorem 5.5]{Di7}
\end{proof}
%%%%%%%%%%%%%%%%%%%%%%%%%%%%%%%%%%%%%%%%%%%%%%%%%%%%%%%%%%%%%%%%%%%%%%%%%%%%%%%%%%%%%%%%%%%

\noindent
{\bf Acknowledgments.} This work was supported by Vietnam National Foundation for Science and 
Technology Development (NAFOSTED).
%%%%%%%%%%%%%%%%%%%%%%%%%%%%%%%%%%%%%%%%%%%%%%%%%%%%%%%%%%%%%%%%%%%%%%%%%%%%%%%%%%%%%%%%%%%

\end{document}